\documentclass[psamsfonts]{amsart}

\usepackage{amssymb,amsfonts}
\usepackage[all,arc]{xy}
\usepackage{enumerate}
\usepackage{mathrsfs}
\usepackage{amsmath}
\usepackage{bbm}
\usepackage{tikz}
\usepackage{pgfplots}
\usepackage{hyperref}
\usepackage[alphabetic]{amsrefs}
\usepackage{mathtools}
\usepackage{thmtools}

\newtheorem{thm}{Theorem}[section]
\newtheorem{cor}[thm]{Corollary}

\newtheorem{lem}[thm]{Lemma}
\newtheorem{conj}[thm]{Conjecture}

\theoremstyle{definition}

\newtheorem{claim}[thm]{Claim}

\theoremstyle{remark}
\newtheorem{rem}[thm]{Remark}

\newcommand{\Q}{\mathbb{Q}}

\newcommand{\e}{\varepsilon}

\newcommand{\Z}{\mathbb{Z}}
\newcommand{\N}{\mathbb{N}}

\renewcommand{\mod}{\text{ mod }}

\renewcommand{\empty}{\emptyset}

\newcommand{\ch}{\mathbbm{1}}
\newcommand{\vertiii}[1]{{\left\vert\kern-0.25ex\left\vert\kern-0.25ex\left\vert #1 
		\right\vert\kern-0.25ex\right\vert\kern-0.25ex\right\vert}}

\renewcommand{\Re}{\text{Re} }

\renewcommand{\S}{\mathcal{S}}

\newcommand{\sgn}{\text{ sign}}

\begin{document}
	\title{A Counterexample To Hildebrand's Conjecture On Stable Sets}
	\author{Redmond McNamara}
	\address{University of Crete, Department of  Mathematics and Applied Mathematics, Heraklion, Greece}
	\maketitle
\begin{abstract}
	We provide a counterexample to a conjecture of Hildebrand which states that if $\S$ has positive lower density and is stable i.e. for all $d$, $n$ is in $\S$ if and only if $dn$ is in $\S$ except on a set of density $0$ then $\S \cap (\S+1) \cap (\S+2)$ has positive lower density and in particular is nonempty. We further show there exists a stable set of density $1 -\frac{1}{q-1}$ such that $\S \cap \cdots \cap (\S + q -1) = \empty$ when $q$ is a prime matching a bound proven by Hildebrand. Finally, we construct a function $f \colon \N \rightarrow \{\pm 1\}$ such that $f(pn) = -f(n)$ for all but a $0$ density set of $n$ depending on the prime $p$ but which fails the analogues of Sarnak and Chowla's conjectures.
\end{abstract}

	\section{Introduction}

Erd\H{o}s (\cite{Erdos1976}) asked whether there are infinitely many twin smooth numbers, numbers $n$ and $n+1$ all of whose prime factors are less than $n^\e$ for some fixed $\e > 0$. Hildebrand (\cite{Hildebrand1}) solved the question, showing not only that there are infinitely many twin smooth numbers, but that the set of such numbers has positive lower density. In fact Hildebrand showed something much more general.

	We say a set $\S$ is stable if for all $p$, for all but a zero density set of $n$, $n$ is in $\S$ if and only if $pn$ is in $\S$ i.e.
	\begin{equation}\label{whatisstable}
		d(\{ n \in \N \colon \ch_{n \in \S} \neq \ch_{pn \in \S} \}) = 0
	\end{equation}
	We remark that knowing \eqref{whatisstable} for primes $p$ implies the same for all natural numbers. We also remark that examples of stables sets include:
	\begin{enumerate}
		\item $n^\e$-smooth numbers $n$ for any fixed value of $\e$
		\item any set of density $0$
		\item  the set of numbers $n$ for which $\frac{\Omega(n) - \log\log n}{\sqrt{\log \log n}}$ is between 100 and 101
		\item the set of numbers $n$ which have more prime factors which are $1 \mod 10$ than $9 \mod 10$ and more prime factors which are $1 \mod 3$ than $2 \mod 3$
		\item  the intersection of any finite number of stable sets.
	\end{enumerate} 
Hildebrand proved the following remarkable conjecture of Balog:
	\begin{thm}[\cite{Hildebrand1}] \label{stablesetsthm1}
		Let $\S$ be a stable set with positive lower density. Then
		\[
		\underline{d}(\S \cap (\S+1)) > 0.
		\]
	\end{thm}
It is perhaps worth remarking that \cite{MR1619809} later found a specific infinite sequence of $k$ consecutive smooth numbers for any $k$ and \cite{Joni} show that with logarithmic density the probability that $n$ and $n+1$ are smooth is the product of the probabilities. Hildebrand (\cite{Hildebrand2}) was also able to show to following generalization of Theorem \ref{stablesetsthm1}
\begin{thm}[\cite{Hildebrand2}]
	Let $k$ be a natural number and let $\S$ be a stable set with 
	\[
	\underline{d}(\S) > 1 - \frac{1}{k-1}.
	\]
	Then for any integers $h_1, \ldots, h_k$
	\[
	\underline{d}((\S + h_1) \cap \cdots \cap (\S + h_k)) > 0.
	\]
\end{thm}
This led Hildebrand to conjecture
		\begin{conj}[\cite{Hildebrand2}]\label{conj}
		Let $\S$ be a stable set with positive lower density. Then for any $k$
		\[
		\underline{d}(\S \cap (\S+1) \cap (\S+2 )\cap \cdots \cap (\S+(k-1))) > 0. 
		\]
	\end{conj}
	A weaker version of this conjecture could be
	\begin{conj}\label{weakconj}
		Let $\S$ be a stable set with positive lower density. Then for any $k$
		\[
		\S \cap (\S+1) \cap (\S+2) \cap \cdots \cap (\S+(k-1)) \neq \emptyset. 
		\]
	\end{conj}
	\begin{rem}\label{density0}
		Note that this is equivalent to the conjecture that the \emph{upper} density of $\S \cap (\S+1) \cap (\S+2) \cap \cdots \cap (\S+(k-1))$ is positive. Otherwise, given a set where the density is $0$, we could delete the set $\S \cap \cdots \cap (\S +k)$ from $\S$ and since that would only affect a set of density $0$ it would not change whether $\S$ was a stable set.
	\end{rem}
Tao and Ter\"av\"ainen (\cite{TJ}) showed a number of variants of Conjecture \ref{conj} hold. For example if $\S$ is uniformly distributed in small intervals and $\underline{d}(\S) > 1 - \frac{1}{k-\frac{4}{3} + o(1)}$ then the conjecture holds. In this note we prove
\begin{thm} \label{mainthm}
	Conjecture \ref{weakconj} fails for $k = 3$ for a set of density $\frac{1}{2}$. 
\end{thm}
We will do this by constructing a set $\S$ for which the intersection has density zero. Applying Remark \ref{density0}, one can obtain a set where the intersection is actually empty. More generally, we will show
\begin{thm} \label{mainthmpart2}
	Conjecture \ref{weakconj} fails for any prime $k$ for a set of density $1 - \frac{1}{k-1}$. 
\end{thm}
We remark that this perfectly matches Hildebrand's bound, meaning this is essentially best possible.

We also give a construction which may be of interest to those studying multiplicative functions. Note that a stable set can be thought of as a binary valued function $f$ satisfying $f(pn) = f(n)$ for all but a density $0$ set of $n$ depending on $p$. We can likewise use a similar construction to prove the following.
\begin{thm}\label{liouvillelike}
	There exists a function $f \colon \N \rightarrow \{\pm 1\}$ such that, for any prime $p$, the set
	\[
	\{ n \in \N \colon f(n) \neq -f(pn) \}
	\]
	has density $0$ and for which
	\[
	\frac{1}{x} \sum_{n \leq x} f(n) f(n+1) < -\e < 0
	\]
	for some $\e$ and for all $x$ sufficiently large. Furthermore
	\[
	\lim_{H \rightarrow \infty} \lim_{x \rightarrow \infty} \frac{1}{x} \sum_{n \leq x}  \left| \frac{1}{H} \sum_{h \leq H} f(n) \ch_{n = 1 \mod 3} \right|  = \frac{1}{3}.
	\]
\end{thm}
We remark that the construction is somewhat similar to constructions in \cite{MR3435814} and \cite{MRT}.
The property that $f(pn) = -f(n)$ except on a set of density $0$ means $f$ behaves somewhat like the Liouville function, the completely multiplicative function $\lambda$ such that $\lambda(pn) = -\lambda(n)$ for any prime $p$ and any $n$. The Chowla and Sarnak conjectures for the Liouville function suggest that $\lambda$ does not correlate with any of its shifts nor with any entropy $0$ sequence. However, Theorem \ref{liouvillelike} suggests that there exists a function $f$ such that $f(pn) = - f(n)$ 100\% of the time but which fails the Chowla and Sarnak conjectures. Thus, it is not enough to use that $f(pn) = - f(n)$ 100\% of the time in order to prove those conjectures. In particular, there is a method introduced to the field by Bourgain, Sarnak and Ziegler (\cite{BSZ}) sometimes called the Katai criterion, the Turan-Kubilius method, or some combination of the six names which has been used to show many instances of the Sarnak conjecture. And yet, the Bourgain-Sarnak-Ziegler-Katai-Turan-Kubilius method works for any function which satisfies $f(pn) = - f(n)$ 100\% of the time. Therefore, Theorem \ref{liouvillelike} rules out this method by itself and others like it when it comes to proving the full Sarnak or Chowla conjectures.

\section*{Notation}
We use $\underline{d}$, $\overline{d}$ and $d$ for lower density, upper density and density respectively. We say $x = y + O(z)$ if there is a constant $C$ such that $|x - y| \leq C |z|$. We say $x = o_{k \rightarrow \infty}(y)$ if $\lim_{k\rightarrow\infty} \frac{x}{y} = 0$. We say $x \lesssim y$ if there exists a constant $C$ such that $x \leq C y$. We use $[N]$ for $\{1, 2, \ldots, N\}$ and $[a,b]$, $(a,b]$ and $(a,b)$ for closed, half-open and open intervals. We set $e(t) = e^{2\pi i t}$. Let $\ch_{x \in A}$ denote the indicator function which is $1$ if $x$ is in $A$ and $0$ otherwise.

\section*{Acknowledgments}
I would like to greatly thank Terence Tao for suggesting this problem and sharing his thoughts, in particular that near counterexamples from his work with Matomaki and Radziwill might be useful (see \cite{MRT}). I would also like to thank Joni Ter\"av\"ainen and Nikos Frantzikinakis for their comments on earlier drafts. \thanks{Supported by the research grant   ELIDEK-HFRI-NextGenerationEU-15689}.

\section{Proof of the Main Theorem}\label{mainsec}

The basic idea is to choose $T$ such that 
\[
\begin{cases}
p^{iT} \approx 1 & \text{ whenever } p = 1 \mod 3 \\
p^{iT} \approx -1 & \text{ whenever } p = -1 \mod 3
\end{cases}
\]
and then define our set $\S$ by
	\begin{equation} \label{pequation}
\begin{cases}
	3^\ell \cdot n \in \S & \text{ whenever } n = 1 \mod 3 \text{ and } \Re(n^{iT}) > 0 \\
	3^\ell \cdot n \in \S & \text{ whenever } n = -1 \mod 3 \text{ and } \Re(n^{iT}) < 0 \\
	3^\ell \in \S & \text{ for all } \ell.
\end{cases}
\end{equation}
Since $n^{iT} \approx (n+1)^{iT}$ for large $n$, for most $n$ either $\Re(n^{iT}) > 0, \Re((n+1)^{iT}) > 0$ and $\Re((n+2)^{iT}) > 0$, one of which is $-1 \mod 3$ and therefore not in $\S$ or $\Re(n^{iT}) < 0, \Re((n+1)^{iT}) < 0$ and $\Re((n+2)^{iT}) < 0$, one of which is $-1 \mod 3$ and therefore not in $\S$. Thus, $\S \cap (\S+1) \cap (\S+2)$ will have density $0$. Furthermore, for any $p = 1\mod 3$ we have $(pn)^{iT} = p^{iT} n^{iT} \approx n^{iT}$ by choice of $T$, so for most $n$ multiplying by $p$ does not change whether $n$ is $1$ or $-1 \mod 3$ and also does not change whether $\Re(n^{iT}) > 0$. Thus, $n$ is in $\S$ if and only if $pn$ is in $\S$ unless $\Re(n) \approx 0$.  Similarly for $p = -1 \mod 3$, $\Re((pn)^{iT}) \approx - \Re(n^{iT})$ so for most $n$ not divisible by $3$ multiplying by $p$ swaps whether $n$ is $1$ or $-1 \mod 3$ and also swaps whether $\Re(n^{iT}) > 0$ so again $n$ is in $\S$ if and only if $pn$ is in $\S$.

The main difficulty with this solution is that for a given finite $T$ we can only ensure \eqref{pequation} holds for finitely many $p$ and with a small but positive error term. The rest of the paper will be dealing with these difficulties by allowing $T$ to grow in such a way that does not prevent $\S$ from being a stable set.

First, we need a few lemmas.

	\begin{lem}\label{ChooseT}
	For all natural numbers $D$, all real numbers $\e, T_0 > 0$  and all functions $f \colon [D] \rightarrow S^1$ there exists $T > T_0$ such that for all $p \leq D$
	\[
	| p^{iT} - f(p) | \leq \e.
	\]
\end{lem}
\begin{proof}
	Classically, the points $\{ \log p \colon p \leq D \}$ are linearly independent over $\Q$. (Using logarithm rules and exponentiating both sides, any linear dependence would produce a way of writing $1$ as a product of primes). Thus $\{T \log 2, T \log 3, T \log 5, \ldots \}$ equidistributes in the torus so the set for which $| p^{iT} - f(p) | \leq \e$ for all $p \leq D$, which has size $\e^{O(\pi(D))}$ where $\pi(D)$ is the number of primes less than $D$, is visited by some sequence of $T$ tending to infinity. Choosing some $T > T_0$ completes the proof.
\end{proof}
\begin{restatable}{lem}{almosteq}
	\label{almosteq}%
	Let $T_1$ and $T_2$ be real numbers greater than $2$ and let $N \geq 2, a_1, a_2$ and $m$ be nonnegative integers with $0 \leq a_i < m$. Then
	\begin{align*}
	\frac{|\{ n \leq N \colon n^{iT_1} \in e([\frac{a_1}{m}, \frac{a_1+1}{m})) \text{ and } n^{iT_2} \in e([\frac{a_2}{m}, \frac{a_2+1}{m})) \}|}{N} \\ = \frac{1}{m^2} + O\left( \frac{1}{T_1} + \frac{T_1 \log T_1}{T_2} + \frac{T_2 \log T_1}{N} \right).
	\end{align*}
\end{restatable}

%\begin{lem}\label{almosteq}
%	Let $T_1$ and $T_2$ be real numbers greater than $1$ and let $N \geq 2, a_1, a_2$ and $m$ be nonnegative integers with $0 \leq a_i < m$. Then
%	\begin{align*}
%	\frac{|\{ n \leq N \colon n^{iT_1} \in e([\frac{a_1}{m}, \frac{a_1+1}{m})) \text{ and } n^{iT_2} \in e([\frac{a_2}{m}, \frac{a_2+1}{m})) \}|}{N} \\ = \frac{1}{m^2} + O\left( \frac{1}{T_1} + \frac{T_1 \log T_1}{T_2} + \frac{T_2 \log T_1}{N} \right).
%	\end{align*}
%\end{lem}
We postpone the proof to Section \ref{lemmaproof}.
\begin{rem}
	We remark that the error term could be improved but what is really important for us is that the error term is $o(1)$ in the regime $T_1 \ll T_2 \ll N$.
\end{rem}
\begin{rem}
	We remark that for fixed $T_i$ the error term cannot go to $0$ as $N \rightarrow \infty$. For example, suppose we $T_1 = 1$, $T_2 = \frac{1}{2 \log 2}$, $m = 2$ and $a_2 = 0$ so that $(2n)^{iT_2} = -n^{iT_2}$. Then for some choices of $N$ we will have the following points contained in our interval:
	\[
	\left(\frac{N}{2}, N\right] \cup \left(\frac{N}{8}, \frac{N}{4}\right] \cup 	\left(\frac{N}{32}, \frac{N}{16} \right] \cup \cdots
	\]
	which by the geometric series formula should have size $\approx \frac{2}{3} N$. Similarly, for some choices of $N$ the complement will be our interval which will have size approximately $\frac{1}{3}N$.
\end{rem}

It is perhaps worth remarking that this gives the following corollary which extends Lemma \ref{almosteq} to arbitrary real intervals, although this is not really necessary for our main argument.
\begin{cor}\label{realcor}
		Let $T_1$ and $T_2$ be real numbers greater than $1$ and let $N \geq 2$ and $m$ be natural numbers. Let $I_1$ and $I_2$ be subsets of $[0,1]$. Then
	\begin{align*}
		\frac{|\{ n \leq N \colon n^{iT_1} \in e(I_1) \text{ and } n^{iT_2} \in e(I_2) \}|}{N} \\ = |I_1||I_2| + O\left( \frac{1}{m} \right)  + m \cdot O\left( \frac{1}{T_1} + \frac{T_1 \log T_1}{T_2} + \frac{T_2 \log T_1}{N} \right).
	\end{align*}
\end{cor}
Again we remark that the error term could be optimized more and the important thing is that it is $o(1)$ in the regime $m \ll T_1 \ll T_2 \ll N$. We also remark that the optimal choice of $m$ for the error term above is the square root of the original error term. The result essentially follows from chopping up the interval from $[0,1]$ into pieces of size $\frac{1}{m}$. Except for at most $O(1)$ many pieces, each piece is contained or disjoint from each $I_j$. We leave the details to the interested reader.

Now we can finally return to the proof of Theorem \ref{mainthm}. Let $D_1, D_2, D_3, \ldots$ be an increasing sequence of natural numbers and $\e_1, \e_2, \e_3, \ldots$ be a decreasing sequence of positive real numbers approaching $0$. Let $J_k = \lceil \frac{2}{\e_k} \rceil + 1$. By Lemma \ref{ChooseT}, for all real numbers $T_0$, there exists $T_k \geq T_0$ such that for all $p \leq D_k$, 
\begin{equation}\label{defT}
\begin{cases}
	|p^{iT_k} - 1| \leq \e_k & \text{ whenever } p = 1 \mod 3 \\
	|p^{iT_k} + 1| \leq \e_k & \text{ whenever } p = -1 \mod 3.
\end{cases}
\end{equation}
We can also fix $t_k \geq T_k$ such that
\begin{equation}\label{deft}
| p^{it_k} - 1 | \leq \e_k
\end{equation}
for all $p \leq D_k$.

We will fix two sequences of real numbers $T_k$ and $t_k$ and a sequence of natural numbers $N_{k,0}, \ldots, N_{k, J_k}$ such that for each $k$, $T_k$ is sufficiently large depending on $\e_k$ and $D_k$, $t_k$ is sufficiently large depending on $T_k$, $D_k$ and $\e_k$ and $N_{k,j}$ is sufficiently large depending on $T_k$, $T_{k+1}$, $t_k$, $\e_k$ and $N_{r,s}$ for $r < k$ and any $s$ or for $r = k$ and any $s < j$. Thus
\[
D_k, \frac{1}{\e_k} \ll T_k \ll t_k \ll T_{k+1} \ll N_{k,0} \ll N_{k,1} \ll \cdots \ll N_{k,J_k} \ll N_{k+1,1}.
\]
For convenience define $N_k = N_{k, 0}$, $N_{k,J_k+1} = N_{k+1}$ and $N_0 = 1$. The precise nature of how big each variable needs to be depending on the last could be worked out by the interested reader but for instance
\begin{enumerate}
\item $\frac{1}{T_k} + \frac{T_k \log T_k}{t_{k}} + \frac{T_k \log T_{k+1}}{N_k}$ tends to $0$ as $k$ tends to infinity.  \label{densityhappens}
\item $N_{k,j}  D_k < N_{k,j+1}$ and $N_{k,J_k}  D_k< N_{k+1,0}$.  \label{jumpsize}
\item In particular, we may assume  $N_{k,j}  2 < N_{k,j+1}$ which is enough to imply 
\[
d(\{n \colon |n - N_{k,j}| < 3 \text{ for some } k, j \}) = 0.
\]
\label{doubling}
\end{enumerate}
Let $n$ be a natural number in the interval $[N_{k,j}, N_{k,j+1})$ or in the interval $[N_{k,j}, N_{k+1,0})$ if $j = J_k$ and suppose $n$ is not divisible by $3$. Let $\alpha(n) = (\e_k \cdot j)- 1$. Let 
\[
\chi(n) =\begin{cases}
	+1 & \text{ if } n = 1 \mod 3 \\
	-1 & \text{ if } n = -1 \mod 3 \\
	0 & \text{ if } n = 0 \mod 3.
\end{cases}
\] 
Then we say $3^\ell n$ is in $\S$ for all $\ell$ if and only if
\[
\begin{cases}
\chi(n) \cdot \sgn( \Re( n^{iT_k} ) ) > 0 & \text{ and } \Re( n^{it_k} ) \geq \alpha(n) \text{ or } \\
\chi(n) \cdot \sgn( \Re( n^{iT_{k+1}} ) ) > 0 &  \text{ and }  \Re( n^{it_k} ) < \alpha(n)
\end{cases}
\]

\begin{rem}
	Whether you use the condition $\chi(n) \cdot \sgn( \Re( n^{iT_k} ) ) > 0 $ or $\chi(n) \cdot \sgn( \Re( n^{iT_{k+1}} ) ) > 0 $ depends whether $ \Re( n^{it_k} ) \geq \alpha(n)$. 
	When $n$ is in $[N_{k,0}, N_{k,1})$ and $\alpha(n) = -1$, then the conditions that $\Re( n^{it_k} ) \geq \alpha(n)$ is trivially satisfied and when $\alpha(n) > 1$ it is necessarily false.  Thus, by slowly changing $\alpha$ we slowly ``turn on" and ``turn off" the restrictions that $\chi(n) \Re(n^{iT_{k}}) > 0$ and $\chi(n) \Re(n^{iT_{k+1}}) > 0$.
\end{rem}
Now we verify the necessary claims about $\S$.

\begin{claim}\label{Sisstable}
	$\S$ is stable.
\end{claim}
Let $p$ be a prime. If $p = 3$ then by examining the definition of $\S$, $n$ is in $\S$ if and only if $3n$ is in $\S$ so assume $p \neq 3$. Let $k$ be such that $p \leq D_k$ which exists since $D_k$ tends to infinity by assumption. Consider the case of some $n$ in $[N_{k,j}, N_{k,j+1})$ which is not divisible by $3$. Then if $j < J_k$, by \eqref{jumpsize} either $pn$ is in $[N_{k,j}, N_{k,j+1})$ or in $[N_{k,j+1}, N_{k,j+2})$. Observe that, by \eqref{defT}, we have
\begin{align*}
	\chi(pn) \Re((pn)^{iT_k})  = & \Re(n^{iT_k} p^{iT_k} \chi(p)) \\
	= & \Re(n^{iT_k}(1 + O(\e_k) )) \\
	= & \Re(n^{iT_k}) + O(\e_k)
\end{align*}
and similarly
\begin{align*}
	\chi(pn) \Re((pn)^{iT_{k+1}})  = & \Re(n^{iT_{k+1}} p^{iT_{k+1}} \chi(p) )\\
	= & \Re(n^{iT_{k+1}} ) + O(\e_k)
\end{align*}
and also because $\alpha(pn) = \alpha(n) + O(\e_k)$ and \eqref{deft} we have
\begin{align*}
	 \Re((pn)^{it_k})  - \alpha(pn) = & \Re(n^{it_k} p^{it_k}) - \alpha(n)  + O(\e_k) \\
	= & \Re(n^{it_k})(1 + O(\e_k) ) - \alpha(n)  + O(\e_k).
\end{align*}
Therefore, each of these expressions can only change signs if one of $n^{iT_k}$, $n^{iT_{k+1}}$, or $n^{it_k}$ is in one of two small intervals of length $O(\e_k)$.
By Lemma \ref{almosteq} we can estimate the probability that any of these happen. In particular by \eqref{densityhappens} that probability goes to $0$. 

The other important case to consider is what happens if $n$ is in $[N_{k,2J_k}, N_{k+1,0})$ and $pn$ is in $[N_{k+1,0}, N_{k+1,1})$. In this case, notice that the $\Re(n^{it_k}) < \alpha_k(n)$ for $n$ and the $\Re((pn)^{it_k}) \geq \alpha_{k+1}(pn)$  actually play no role because each is of the form $\Re(z) \geq -1$ or $\Re(z) \leq 1$ which are trivially true. 
Thus as before $n$ is in $\S$ if and only if $pn$ is in $\S$ unless $\Re(n^{iT_{k+1}}) = O(\e_k)$ which is rare. 

So far, we only handled the case of $n$ not divisible by $3$. Let's see how this allows us to estimate the number of $n$ such that $\ch_{n \in \S} \ne \ch_{pn \in \S}$ in general. Let $\ell$ be a natural number. Let $E(A,B)$ denote the number of $n$ between $A$ and $B$ such that $\ch_{n \in \S} \ne \ch_{pn \in \S}$ and let $E_0(A,B)$ be the number of such $n$ which are not divisible by $3$. Then we can write $E(1,N)$ as a kind of Ces\`aro average:
\begin{equation}\label{cesaro}
E(1,N) \leq E_0\left(\frac{N}{3}, N\right) + 2 E_0\left(\frac{N}{9}, \frac{N}{3}\right) + \cdots + \ell E_0\left(\frac{N}{3^\ell}, \frac{N}{3^{\ell-1}}\right) + O\left(\frac{N}{3^\ell}\right)
\end{equation}
since each $n \geq N 3^{-j}$ can show up $j$ times until $3^j n \geq N$ i.e. we have the expression $\ch_{3^h n \in \S} \ne \ch_{p3^h n \in \S}$ for each $h \leq j$. The remaining $O\left(\frac{N}{3^\ell}\right)$ accounts for all numbers divisible by $3^\ell$ and therefore not included in our count and also for all the numbers between $1$ and $\frac{N}{3^\ell}$. We conclude that if the frequency of $\ch_{n \in \S} \ne \ch_{p n \in \S}$ in $[3^{-j}N, 3^{-j+1}N)$ for $j \leq \ell$ and for $n$ not divisible by $3$ is at most $\e_k$ then at a minimum
\[
E(1,N) \leq \ell \e_k N + O\left(\frac{N}{3^\ell}\right).
\]
(Of course we can get a substantially better error term but this is all that is need for our purposes). Since this holds for all $k$ for $N$ sufficiently large and $\ell$ fixed, we conclude that 
\[
\overline{d}(\{n \colon \ch_{n \in \S} \ne \ch_{p n \in \S}\}) \lesssim 3^{-\ell}
\]
for all $\ell$. Since this holds for all $\ell$ we conclude 
\[
d(\{n \colon \ch_{n \in \S} \ne \ch_{p n \in \S}\}) = 0.
\]
Since $p$ was arbitrary, $\S$ is stable.

\begin{claim}\label{intersectionisempty}
	\[
	d(\S \cap (\S+1) \cap (\S+2)) = 0.
	\]
\end{claim}

We also note that by Taylor expansion for all real numbers $T$ and natural numbers $h$ and $n$ such that $T h \ll n$
\begin{align*}
	(n+h)^{iT} = & n^{iT} \exp\left(iT \log\left(1+\frac{h}{n}\right)\right) \\
	= & n^{iT} \cdot \left( 1 + O\left(T \log\left(1+\frac{h}{n}\right)\right) \right)  \\
	= & n^{iT} \cdot \left( 1 + O\left( \frac{Th}{n} \right) \right).
\end{align*}
Thus, for all $n \gg T_{k+1}$, at least one of the following three things happens
\begin{align*}
\Re(n^{iT_{k+1}}) > \alpha(n), \Re((n+1)^{iT_{k+1}}) > \alpha(n), & \text{ and } \Re((n+2)^{iT_{k+1}}) > \alpha(n), 
\intertext{ or }
\Re(n^{iT_{k+1}}) < \alpha(n), \Re((n+1)^{iT_{k+1}}) < \alpha(n), & \text{ and } \Re((n+2)^{iT_{k+1}}) < \alpha(n),
\intertext{ or }
\Re(n^{iT_{k+1}}) = \alpha(n) + O\left( \frac{T_{k+1}h}{n} \right).
\end{align*}
The last condition happens with probability $O\left( \frac{1}{t_{k}} + \frac{t_{k} \log t_{k}}{N_{k,j}} \right)$ which in particular tends to $0$ as $k$ tends to infinity. Similarly, for all $n$ outside of a set of density tending to $0$,
\begin{equation*}\label{resgn1}
\sgn\left(\Re (n^{iT_{k}})\right) = \sgn\left(\Re((n+1)^{iT_{k}}))\right) = \sgn\left(\Re((n+2)^{iT_{k}})\right)
\end{equation*}
and 
\begin{equation*}\label{resgn2}
\sgn\left(\Re(n^{iT_{k+1}})\right) = \sgn\left(\Re((n+1)^{iT_{k+1}})\right) = \sgn\left(\Re((n+2)^{iT_{k+1}})\right)
\end{equation*}
 However, if, for instance, both
 \[
 \Re((n+j)^{it_k}) > \alpha(n) \text{ and } \Re((n+j)^{iT_k}) > 0
 \]
 for each of $j = 0,1$ and $2$, then one of those three terms will have $n = -1 \mod 3$ and therefore that $n$ will not be in $\S$.
 Similarly, if, 
 \[
 \Re((n+j)^{it_k}) > \alpha(n) \text{ and } \Re((n+j)^{iT_k}) < 0
 \]
 for each of $j = 0,1$ and $2$, then one of those three terms will have $n = 1 \mod 3$ and therefore that $n$ will not be in $\S$. Likewise in the other two cases.

\begin{claim}\label{densityquarter}
	\[
	d(\S) = \frac{1}{2}.
	\]
\end{claim}
For each for all $n$ between $N_k$ and $N_{k+1}$ not divisible by $3$, $n$ is in $\S$ if  
\[
\Re(\chi(n) n^{it_k}) \geq 0 \text{ and } \Re(n^{iT_k}) > 0
\]
By Lemma \ref{almosteq}, this happens with probability $\frac{1}{4} + O\left(\frac{1}{T_k} + \frac{T_k \log T_k}{t_k} + \frac{t_k \log t_{k}}{N_k}\right)$. Alternately, $n$ is in $\S$ if  
\[
\Re(\chi(n) n^{it_k}) > 0 \text{ and } \Re(n^{iT_{k+1}}) > 0
\]
which again by Lemma \ref{almosteq} happens with probability $\frac{1}{4} + O\left(\frac{1}{t_k} + \frac{t_k \log t_k}{T_{k+1}} + \frac{T_{k+1} \log t_{k}}{N_k}\right)$. The error term in both cases goes to zero by \eqref{densityhappens}. Again, summing the contributions from powers of $3$ as in \eqref{cesaro} yields the claim.

Now we can conclude the proof of Theorem \ref{mainthm} modulo the proof of Lemma \ref{almosteq}. By Claims \ref{Sisstable}, \eqref{intersectionisempty} and \eqref{densityquarter}, $\S$ is a stable set where $\S \cap (\S+1) \cap (\S+2)$ has density $0$ and such that $d(\S) = \frac{1}{2}$. Thus $\S$ satisfies the hypothesis of Theorem \ref{mainthm}.

\section{Proof of Lemma \ref{almosteq}} \label{lemmaproof}
This section is devoted to the proof of the following lemma:
\almosteq*
	We prove the lemma in three steps.
\begin{claim}\label{claim1}
	The map defined by $[0,1] \ni x \mapsto (Nx)^{iT_1}$ is almost equidistributed in the sense that
	\[
	\left| \left\{ x \in [0,1] \colon (Nx)^{iT_1} \in e\left(\left[\frac{a_1}{m}, \frac{a_1+1}{m}\right)\right) \right\} \right| = \frac{1}{m} + O\left(\frac{1}{T_1}\right)
	\]
\end{claim}
\begin{proof}[Proof of Claim \ref{claim1}]
	Note that for any integer $b$ that $(Nx)^{iT_1}$ is in $e\left(\left[\frac{b}{m}, \frac{b+1}{m}\right)\right)$ if and only if
	\[
	\log x \in \frac{1}{T_1} \left( \left[\frac{b}{m}, \frac{b+1}{m}\right) + \ell \right) - \log N
	\]
	for some integer $\ell$. We can partition the unit interval into pieces such that $(Nx)^{iT_1}$ lies in $e\left(\left[\frac{b}{m}, \frac{b+1}{m}\right)\right)$ as follows:
	\[
	[0,1] = \bigcup_{0 \leq b < m} \bigcup_{\ell} \left\{ x \colon 	\log x \in \frac{1}{T_1} \left( \left[\frac{b}{m}, \frac{b+1}{m}\right) + \ell \right) - \log N \right\} \cap [0,1].
	\]
	For each $0 \leq b < m$, the map which multiplies every point by $e^\frac{2\pi(a-b)}{mT_1}$ is a bijection from the set 
	\begin{align*}& \left\{  x \colon 	\log x \in \frac{1}{T_1} \left( \left[\frac{b}{m}, \frac{b+1}{m}\right) + \ell \right) - \log N \right\}
		\intertext{ to the set
		}
		& \left\{  x \colon 	\log x \in \frac{1}{T_1} \left( \left[\frac{a_1}{m}, \frac{a_1+1}{m}\right) + \ell \right) - \log N \right\}
		\intertext{ and multiplies lengths by $e^\frac{2\pi(a-b)}{mT_1} = 1 + O\left( \frac{1}{T_1} \right)$ by the Taylor expansion for $e^x$. This almost gives a bijection on the sets in our partition } 
		\bigcup_{\ell} & \left\{  x \colon 	\log x \in \frac{1}{T_1} \left( \left[\frac{b}{m}, \frac{b+1}{m}\right) + \ell \right) - \log N \right\} \cap [0,1]
	\end{align*}
	except that we have intersected with the interval $[0,1]$. However, this only effects those $x$ which are in the interval $[e^\frac{-1}{T_1},1]$ which again has length $1 + O\left(\frac{1}{T_1}\right)$. Thus
	\begin{align*}
		\left| \bigcup_{0 \leq b < m} \bigcup_{\ell}  \left\{ x \colon 	\log x \in \frac{1}{T_1} \left( \left[\frac{b}{m}, \frac{b+1}{m}\right) + \ell \right) - \log N \right\} \cap [0,1] \right| &= \\
		m \left( \left| \bigcup_{\ell}  \left\{  x \colon 	\log x \in \frac{1}{T_1} \left( \left[\frac{a_1}{m}, \frac{a_1+1}{m}\right) + \ell \right) - \log N \right\} \cap [0,1] \right| +O\left( \frac{1}{T_1} \right) \right) &=\\
		1 + O\left( \frac{1}{T_1} \right) &.
	\end{align*} 
	Dividing by $m$ completes the proof of Claim \ref{claim1}.
\end{proof}
\begin{claim}\label{claim2}
	The map defined by $[0,1] \ni x \mapsto ((Nx)^{iT_1}, (Nx)^{iT_2})$ is almost equidistributed in the sense that
	\begin{align*}
		& \left\{ x \in [0,1] \colon (Nx)^{iT_1} \in e\left(\left[\frac{a_1}{m}, \frac{a_1+1}{m}\right)\right) \text{ and } (Nx)^{iT_2} \in e\left(\left[\frac{a_2}{m}, \frac{a_2+1}{m}\right)\right) \right\} \\ & =  \frac{1}{m^2} + O\left(\frac{1}{T_1} + \frac{T_1 \log T_1}{T_2}\right)
	\end{align*}
\end{claim}
\begin{proof}[Proof of Claim \ref{claim2}]
	We basically emulate the proof of Claim \ref{claim1} on each interval of the form
	\[
	\left\{ x \colon 	\log x \in \frac{1}{T_1} \left( \left[\frac{a_1}{m}, \frac{a_1+1}{m}\right) + \ell \right) - \log N \right\}.
	\]
	We split each such interval up into pieces
	\begin{align*}
		\bigcup_{0 \leq b < m} \bigcup_{\ell} & \left\{ x \colon 	\log x \in \frac{1}{T_2} \left( \left[\frac{b}{m}, \frac{b+1}{m}\right) + \ell \right) - \log N \right\} \\ \cap \bigcup_{\ell'} &\left\{ x \colon 	\log x \in \frac{1}{T_1} \left( \left[\frac{a_1}{m}, \frac{a_1+1}{m}\right) + \ell' \right) - \log N \right\}.
	\end{align*}
	By the previous argument, each piece of our partition corresponding to some $b$ is the same size up to an error of size $\frac{1}{T_2}$. However now in estimating the full expression
	\[
	\left\{ x \in [0,1] \colon (Nx)^{iT_1} \in e\left(\left[\frac{a_1}{m}, \frac{a_1+1}{m}\right)\right) \text{ and } (Nx)^{iT_2} \in e\left(\left[\frac{a_2}{m}, \frac{a_2+1}{m}\right)\right) \right\}
	\]
	we have one such error for each interval of the form
	\[
	\left\{ x \colon 	\log x \in \frac{1}{T_1} \left( \left[\frac{a_1}{m}, \frac{a_1+1}{m}\right) + \ell' \right) - \log N \right\}.
	\]
	Accepting an error of size $O\left(\frac{1}{T_1}\right)$, we can restrict ourselves to intervals in $\left[\frac{1}{T_1}, 1\right]$. Since we get another interval each time we multiply by $e^\frac{2\pi}{T_1}$ this is precisely $\log_{e^\frac{2\pi}{T_1}} T_1 = \frac{T_1}{2\pi} \log T_1$. Altogether we find
	\begin{align*}
		\Bigg| & \bigcup_{0 \leq b < m}  \bigcup_{\ell}  \left\{ x \colon 	\log x \in \frac{1}{T_2} \left( \left[\frac{b}{m}, \frac{b+1}{m}\right) + \ell \right) - \log N \right\} \\ \cap &\bigcup_{\ell'} \left\{ x \colon 	\log x \in \frac{1}{T_1} \left( \left[\frac{a_1}{m}, \frac{a_1+1}{m}\right) + \ell' \right) - \log N \right\} \cap \left[\frac{1}{T_1},1\right] \Bigg| \\ 
		= m  \Bigg| & \bigcup_{\ell}  \left\{ x \colon 	\log x \in \frac{1}{T_2} \left( \left[\frac{a_2}{m}, \frac{a_2+1}{m}\right) + \ell \right) - \log N \right\} \\ \cap &\bigcup_{\ell'} \left\{ x \colon 	\log x \in \frac{1}{T_1} \left( \left[\frac{a_1}{m}, \frac{a_1+1}{m}\right) + \ell' \right) - \log N \right\} \cap \left[\frac{1}{T_1},1\right] \Bigg|+ O\left( \frac{m T_1 \log T_1}{T_2} \right) \\
	\end{align*}
	which gives Claim \ref{claim2} after dividing by $m$ and applying Claim \ref{claim1}.
\end{proof}
Finally we complete the proof of Lemma \ref{almosteq}.
\begin{proof}
	Recall that for any real numbers $x < y$ the interval $[x,y)$ contains $y-x + \text{Error}$ many integers where $\text{Error} < 1$. By Claim \ref{claim2}, 
	\begin{align*}
		& \left\{ x \in [0,1] \colon (Nx)^{iT_1} \in e\left(\left[\frac{a_1}{m}, \frac{a_1+1}{m}\right)\right) \text{ and } (Nx)^{iT_2} \in e\left(\left[\frac{a_2}{m}, \frac{a_2+1}{m}\right)\right) \right\} \\ & =  \frac{1}{m^2} + O\left(\frac{1}{T_1} + \frac{T_1 \log T_1}{T_2}\right)
	\end{align*}
	Note that we are trying to estimate
	\begin{equation}\label{tempeq}
		\frac{|\{ n \leq N \colon n^{iT_1} \in e([\frac{a_1}{m}, \frac{a_1+1}{m})) \text{ and } n^{iT_2} \in e([\frac{a_2}{m}, \frac{a_2+1}{m})) \}|}{N} 
	\end{equation}
	i.e. which fraction of natural numbers $n$ for which $\frac{n}{N}$ is one of those values $x$ such that $x \mapsto ((Nx)^{iT_1}, (Nx)^{iT_2})$  lands in the set $e\left(\left[\frac{a_1}{m}, \frac{a_1+1}{m}\right)\right) \text{ and } (Nx)^{iT_2} \in e\left(\left[\frac{a_2}{m}, \frac{a_2+1}{m}\right)\right)$. Claim \ref{claim2} shows that \eqref{tempeq} is 
	\[
	\frac{1}{m^2} + O\left(\frac{1}{T_1} + \frac{T_1 \log T_1}{T_2}\right) + \text{ Error }
	\]
	where the error term corresponds to the fact that an each interval of the form
	\[
	\left\{ x \colon \log x \in \frac{1}{T_2} \left( \left[\frac{a_2}{m}, \frac{a_2+1}{m}\right) + \ell \right) - \log N, \in \frac{1}{T_1} \left( \left[\frac{a_1}{m}, \frac{a_1+1}{m}\right) + \ell \right) - \log N \right\} 
	\]
	whose end points we might call $[y,z)$ does not correspond to exactly $N(z-y)$ natural numbers but only up to an error of size at most $1$. Thus, we simply have to estimate the number of intervals. The number of intervals in question is at most the sum of the number of times $n^{iT_1}$ and $n^{iT_2}$ go through a full rotation in the interval $\left[\frac{1}{T_1},1\right]$ and since Lemma \ref{almosteq} is trivial unless $T_1 \ll T_2$  we make restrict attention to the second term. As in the proof of Claim \ref{claim2} the number of times $n^{iT_2}$ goes through a full rotation in the interval $\left[\frac{1}{T_1},1\right]$ is $O\left( T_2 \log T_1 \right)$ which, upon dividing by $N$, completes the proof of Lemma \ref{almosteq} and thus Theorem \ref{mainthm}.
\end{proof}

\section{A Remark on Variants of the Main Theorem}
First, we show the following, giving slightly fewer details because the construction is very similar to the construction in Section \ref{mainsec}:

\begin{thm}
	For any prime number $q$, we can construct a stable set of density $d(\S) = 1 - \frac{1}{q-1}$ for which 
	\[
	\S \cap (\S+1) \cap \dots \cap (\S+q-1)
	\]
	has density 0 (and therefore by Remark \ref{density0} there exists a stable set where the intersection is empty). 
\end{thm}
We remark that this exactly matches Hildebrand's bound in \cite{Hildebrand2}. Thus, any stable set with greater density must have the relevant intersection nonempty.
 
The construction goes as follows. Fix $\chi$ a character on $(\Z / q \Z)^\times$ taking $q-1$ different values. This is possible because $(\Z / q \Z)^\times$ is a cyclic group so one can pick a generator and declare that it maps to a primitive $(q-1)^{th}$ root of unity under $\chi$, which determines such a character. As before we choose two sequences of real numbers: again
\[
| p^{i t_k} - 1 | \leq \e_k
\]
and now
\[
|p^{i T_k} \chi(p) - 1 | \leq \e_k
\]
for all $p \leq D_k$. Again, we can pick some sequence $N_{k,j}$ with the property that
\[
D_k, \frac{1}{\e_k} \ll T_k \ll t_k \ll T_{k+1} \ll N_{k,0} \ll N_{k,1} \ll \cdots \ll N_{k,J_k} \ll N_{k+1,1}
\]
where again $J_k > \frac{2}{\e_k}$.
For $n$ coprime to $q$ and between $N_{k,j}$ and $N_{k,j+1}$, we will say that $n$ is in $\S$ if and only if 
\[
\arg\left( {\chi(n) \cdot n^{iT_\ell}} \right) \not\in \left[ 0, \frac{2 \pi}{q-1} \right]
\]
where $\ell = j$ if $\Re (n^{it_k}) \geq -1 + \e_k \cdot j $ and $\ell = j+1$ if not. We remark that this essentially gives the set constructed in Section \ref{mainsec} when $q = 3$. So that the set $\S$ is still stable under multiplication by $q$, we say $q^m n \in \S$ if and only if $n \in \S$ for any $m$ and any $n$ coprime to $q$. Also as in Section \ref{mainsec},  for any $n$, because the values of $\chi(n), \ldots, \chi(n+q-1)$ include all $(q-1)^{th}$ roots of unity. Therefore, there must exist $h \leq q-1$ such that
\[
\arg\left( {\chi(n+h) \cdot n^{iT_\ell}} \right) \in \left[ 0, \frac{2 \pi}{q-1} \right]
\]
where again $\ell = j$ if $\Re (n^{it_k}) \geq -1 + \e_k \cdot j $ and $\ell = j+1$ if not. We again conclude that 
\[
\arg\left( {\chi(n+h) \cdot (n+h)^{iT_\ell}} \right) \in \left[ 0, \frac{2 \pi}{q-1} \right]
\]
unless 
\[
\arg\left(n^{iT_\ell}\right)  = O\left( \frac{q}{n} \right)
\]
or
\[
\left| \arg\left(n^{iT_\ell}\right) - \frac{2\pi}{q-1} \right| = O\left( \frac{q}{n} \right)
\]
or
\[
\Re (n^{it_k}) + 1 - \e_k \cdot j = O\left( \frac{q}{n} \right).
\]
By a similar argument as before, using Lemma \ref{almosteq}, one can show this happens only on a set of density $0$. Thus, $\S \cap \cdots \cap (\S + q)$ has density $0$. 

Also, just as in Section \ref{mainsec}, for any $p \leq D_k$ and any $n$ between $N_{k,j}$ and $N_{k,j+1}$, we have that
\begin{align*}
\chi(pn) (pn)^{i T_\ell} = &  \ \chi(p) p^{iT_\ell} \chi(n)  n^{i T_\ell} \\
 = &  \ (1 + O(\e_k)) \cdot \chi(n)  n^{i T_\ell} 
\end{align*}
for both $\ell = k$ and $k+1$. Thus, either
\[
 \arg \left( \chi(n)  n^{i T_\ell} \right)  \in \left[ 0, \frac{2 \pi}{q-1} \right] \text{ and }  \arg \left( \chi(pn)  (pn)^{i T_\ell} \right)  \in \left[ 0, \frac{2 \pi}{q-1} \right] 
\]
or 
\[
\arg \left( \chi(n)  n^{i T_\ell} \right)  \not\in \left[ 0, \frac{2 \pi}{q-1} \right] \text{ and }  \arg \left( \chi(pn)  (pn)^{i T_\ell} \right)  \not\in \left[ 0, \frac{2 \pi}{q-1} \right] 
\]
or
$n^{i T_\ell}$ is in a set of size $O(q \e_k)$. Since $q$ is fixed, this goes to $0$ as $k$ tends to infinity and by using Lemma \ref{almosteq} we can show this happens on a $0$ density set. Similarly, 
\[
(pn)^{i t_k} = (1 + O(\e_k)) n^{it_k}
\]
so either 
\[
\Re (n^{it_k}) + 1 - \e_k \cdot j \geq 0 \text{ and } \Re ((pn)^{it_k}) + 1 - \e_k \cdot j \geq 0
\]
or
\[
\Re (n^{it_k}) + 1 - \e_k \cdot j < 0 \text{ and } \Re ((pn)^{it_k}) + 1 - \e_k \cdot j < 0
\]
or $n^{it_k}$ is in an interval of size $O(\e_k)$ which again happens only on a sparse set. 

That the density of $\S$ is $1 - \frac{1}{q-1}$ can be seen because, for any $m$ and any $n = 1 \mod q$, unless $n$ is within $q$ of $N_{j,k}$ for some $j$ and $k$ or $n^{it_k}$ is in an interval of size $O\left(\frac{q}{n} \right)$, then all but one of
\[
q^m n, q^m(n+1), \ldots, q^m (n+q-2)
\]
will be in $\S$. Since the relative density of the multiples of $q$ in $\S$ is the same as the density of $\S$ by stability, the density of $\S$ is $1 - \frac{1}{q-1}$. This completes the construction.

The other interesting construction that can be made with essentially the same method relates to the Liouville function. Notice that a stable set is essentially the same as a binary valued function $f$ for which $f(mn) = 1 \cdot f(n)$ for a full density set of $n$ depending on $m$. But the same methods allow us to construct functions which are ``almost multiplicative" where $1$ is replaced by another number. For example, the widely believed Chowla conjecture suggests that if $\lambda$ is the unique completely multiplicative function for which
\[
\lambda(pn) = - \lambda(n)
\]
for all primes $p$, then
\[
\sum_{n \leq x} \lambda(n) \lambda(n+1) = o(x).
\]
However, we can construct a function for which 
\[
f(mn) = \lambda(m) f(n)
\]
for a full density set of $n$ depending on $m$ but for which the Chowla conjecture fails. This means that it is not enough to use that $f$ is multiplicative 100\% of the time to prove Chowla's conjecture. For example, methods like the Bourgain-Sarnak-Ziegler-Katai-Turan-Kubilius method which only rely on the fact that $f(mn) = \lambda(m) f(n)$ for 100\% of $n$ cannot hope to resolve the Chowla conjecture.

\begin{thm}
	There exists a function $f \colon \N \rightarrow \{\pm 1\}$ such that, for any prime $p$, the set
	\[
	\{ n \in \N \colon f(n) \neq -f(pn) \}
	\]
	has density $0$ and for which
	\[
	\frac{1}{x} \sum_{n \leq x} f(n) f(n+1) < -\e < 0
	\]
	for some $\e$ and for all $x$ sufficiently large. Furthermore
	\[
	\lim_{H \rightarrow \infty} \lim_{x \rightarrow \infty} \frac{1}{x} \sum_{n \leq x}  \left| \frac{1}{H} \sum_{h \leq H} f(n) \ch_{n = 1 \mod 3} \right|  = \frac{1}{3}.
	\]
\end{thm}

The construction proceeds similar to that in Section \ref{mainsec}. For a sequence $\e_k$ tending to $0$ and $D_k$ tending to infinity we again find sequences $T_k$, $t_k$ and $N_{k,j}$ such that 
\[
D_k, \frac{1}{\e_k} \ll T_k \ll t_k \ll T_{k+1} \ll N_{k,0} \ll N_{k,1} \ll \cdots \ll N_{k,J_k} \ll N_{k+1,1}
\]
and
\[
p^{i t_k} = 1 + O(\e_k)
\]
for all $p \leq D_k$ not equal to 3 but now the key difference is that
\[
p^{iT_k} = -\chi(p) +O(\e_k)
\]
where $\chi$ is again the real character modulo 3 i.e. $\chi(n) = 1$ if $n = 1 \mod 3$ and $\chi(n) = -1$ if $n = -1 \mod 3$. For $n$ between $N_{k,j}$ and $N_{k,j+1}$ coprime to $3$, we set
\[
f(n) = \sgn \left( \Re \left( \chi(n) n^{iT_\ell} \right) \right)
\]
where $\ell = k$ if $\Re(n^{it_k}) \geq -1 + j \e_k$ and $k+1$ otherwise. Now because of the added minus sign
\begin{align*}
	f(pn) = & \sgn \left( \Re \left( \chi(pn) (pn)^{iT_\ell} \right) \right) \\
	= & \sgn \left( \Re \left( (-1 + O(\e_k) ) \chi(n) n^{iT_\ell} \right) \right)
	\intertext{
	which equals
	}
	= & - \sgn \left( \Re \left( \chi(n) n^{iT_\ell} \right) \right) \\
	= & - f(n)
\end{align*}
unless one of $n^{iT_k}$, $n^{iT_{k+1}}$ or $n^{it_k}$ are in a set of size $O(\e_k)$ which again happens on a set of density $0$. We set $f(3^m n) = (-1)^m f(n)$. Then $f(n) f(n+1)$ is always equal to $-1$ when $n = 1 \mod 3$ but is equally likely to be $\pm 1$ for $n = 0, 2 \mod 3$, so the two points correlations approach $-\frac{1}{3}$. Furthermore, by Taylor expansion, for any fixed $H$, unless one of $n^{iT_k}$, $n^{iT_{k+1}}$ or $n^{it_k}$ are in a set of size $O\left( \frac{H}{n} \right)$, we have $f(n + h)$ is constant for all $h$ such that $n+h = 1 \mod 3$ giving the desired property.

\end{document}